\documentclass[runningheads, orivec, envcountsame]{llncs}
\usepackage{amsmath,amssymb}
\usepackage{array}
\usepackage{mathrsfs}
\usepackage{stmaryrd}
\usepackage{bm}
\usepackage{enumitem}
\usepackage{graphics}
\usepackage{booktabs}
\usepackage{hhline}
\usepackage{multirow}
\usepackage{color}
\spnewtheorem*{theorem*}{Theorem}{\bf}{\itshape}
\spnewtheorem*{proof*}{}{\itshape}{\upshape}
\spnewtheorem*{corollary*}{Corollary}{\bf}{\itshape}
\newcolumntype{C}[1]{>{\hfil}m{#1}<{\hfil}}
\newcommand\lam[2]{\lambda#1.\,#2}
\DeclareMathOperator{\FV}{FV}
\DeclareMathOperator{\dom}{dom}
\newcommand{\To}{\Rightarrow}
\newcommand{\bzero}{\mathbf{0}}
\newcommand{\bone}{\mathbf{1}}
\newcommand{\Fml}{\mathop{\mathrm{FOFml}}}
\newcommand{\FOFml}{\mathop{\mathrm{FOFml}}}
\newcommand{\ba}{\mathbf{a}}
\newcommand{\bb}{\mathbf{b}}

\newcommand{\ttfunc}[1]{\mathop{\mathrm{t}_{#1}}}
\newcommand{\arity}{\mathop{\mathrm{ar}}}
\newcommand{\ILS}{\mathop{\mathrm{ILS}}}

\newcommand{\CLS}{\mathop{\mathrm{CLS}}}
\newcommand{\sK}{\mathscr{K}}
\newcommand{\sC}{\mathscr{C}}

\newcommand{\FOCLS}{\mathop{\mathrm{FOCLS}}}
\newcommand{\FOCDS}{\mathop{\mathrm{FOCDS}}}
\newcommand{\FOILS}{\mathop{\mathrm{FOILS}}}

\newcommand{\FOSqt}{\mathop{\mathrm{FOSqt}}}

\newcommand{\emptyfunc}{\varnothing}

\newcommand{\CD}{\mathrm{CD}}
\newcommand{\exor}{\oplus}
\newcommand{\barred}{\mathrel{\reflectbox{$\vdash$}}}
\newcommand{\sP}{\mathscr{P}}
\newcommand{\Last}{\mathop{\mathrm{Last}}}
\newcommand{\D}{\mathbf{D}}

\setlist[description]{font=\scshape\mdseries}
\begin{document}
\title{What Kinds of Connectives Cause \\ the Difference between \\ Intuitionistic Predicate Logic \\ and the Logic of Constant Domains?}
\titlerunning{What Kinds of Connectives Cause the Difference?}
\author{Naosuke Matsuda\inst{1}
\and
Kento Takagi\inst{2}\orcidID{0000-0003-3810-9610}}
\authorrunning{N. Matsuda and K. Takagi}
\institute{Department of Engineering, Niigata Institute of Technology, \\
Fujihashi, Kashiwazaki City, Niigata 945-1195, Japan \\
\email{matsuda.naosuke@gmail.com}\\
\and
Department of Computer Science, Tokyo Institute of Technology, \\ Ookayama, Meguro-ku, Tokyo 152-8522, Japan \\
\email{kento.takagi.aa@gmail.com}}
\maketitle         
\begin{abstract}
It is known that intuitionistic Kripke semantics can be generalized so that it can treat arbitrary propositional connectives characterized by truth functions. We extend this generalized Kripke semantics to first-order logic, and study how the choice of connectives changes the relation between intuitionistic predicate logic and the logic of constant domains in terms of validity of sequents. Our main result gives a simple necessary and sufficient condition for the set of valid sequents in intuitionistic predicate logic to coincide with the set of valid sequents in the logic of constant domains.

\end{abstract}
%%%
\section{Introduction}\label{section: Introduction}
\subsection{Background---generalized Kripke semantics}\label{subsection: Background}
\label{subsection: Background}
In \cite{kripke1965semantical}, Kripke provided the intuitionistic interpretation for formulas built out of the usual propositional connectives $\lnot$, $\to$, $\land$ and $\lor$. The notion of validity in intuitionistic logic can be defined with this interpretation. Rousseau~\cite{rousseau1970sequents} and Geuvers and Hurkens~\cite{geuvers2017deriving} extended the intuitionistic interpretation so that it can treat arbitrary propositional connectives characterized by truth functions. Their idea is very simple: when $c$ is a propositional connective and $\ttfunc{c}$ is the truth function associated with $c$, then the value $\| c(\alpha_1, \ldots, \alpha_n) \|_w$ of formula $c(\alpha_1, \ldots, \alpha_n)$ at world $w$ is defined as follows:
  \[
  \| c(\alpha_1, \ldots, \alpha_n )\|_w = 1 \ \text{ if and only if } \ \text{$\ttfunc{c}(\| \alpha_1 \|_v, \ldots, \| \alpha_n \|_v) = 1$ for all $v \succeq w$}.
  \]   
  
It is well-known that the relation between intuitionistic logic and  classical logic changes depending on what kinds of propositional connectives are used. In this paper, we consider an important relation between classical logic and intuitionistic logic---the inclusion relationship between the set of classically valid sequents and the set of intuitionistically valid sequents. For a given set $\sC$ of propositional connectives,  let $\CLS(\sC)$ denote the set of classically valid sequents built out of connetives in $\sC$ and $\ILS(\sC)$ the set of Kripke-valid sequents built out of connectives in $\sC$. Then, $\ILS(\sC) \subseteq \CLS(\sC)$ always holds because classical models can be regarded as one-world Kripke models. However, the converse inclusion depends on $\sC$. For example, $\ILS(\{\lnot\}) \subsetneq \CLS(\{ \lnot \})$ and $\ILS(\{\land, \lor \}) = \CLS(\{ \land, \lor \})$ hold.\footnote{For example, $\lnot \lnot p \To p \in \CLS(\{ \lnot \}) \setminus \ILS(\{ \lnot \})$.}
Then, there arises a natural question: for what $\sC$, does $\ILS(\sC) = \CLS(\sC)$ hold? We answered this question in \cite{kawano2021effect}:
  \begin{theorem*}\label{thm: relation between ILS and CLS}
  $\ILS(\sC) = \CLS(\sC)$ if and only if all $c \in \sC$ are monotonic.
  \end{theorem*}
\subsection{Predicate logic with general propositional connectives}\label{subsection: Predicate logic with truth functional connectives}
Generalized Kripke semantics introduced in \S~\ref{subsection: Background}, as well as classical semantics, can straightforwardly be extended to first-order logic by adding $\forall$ and $\exists$ with the usual interpretations.
Furthermore, in the case of predicate logic, there is an important intermediate logic between classical logic and intuitionistic logic, called the \emph{logic of constant domains}, or $\CD$. $\CD$ is characterized by Kripke models with constant domains.

This paper analyzes how the choice of connectives changes the relation between classical first-order logic, intuitionistic first-order logic and the logic of constant domains in terms of validity of sequents. Let $\FOCLS(\sC)$ be the set of classically valid \emph{predicate} sequents built out of connectives in $\sC$, $\FOILS(\sC)$ the set of Kripke-valid \emph{predicate} sequents built out of connectives in $\sC$ and $\FOCDS(\sC)$ the set of $\CD$-valid \emph{predicate} sequents built out of connectives in $\sC$.\footnote{$\CD$-valid sequents means those sequents that are valid in all constant domain Kripke models.}
Then, by the definitions of each kinds of models, $\FOILS(\sC) \subseteq \FOCDS(\sC) \subseteq \FOCLS(\sC)$ always holds. However, the converse inclusion relationships depend on $\sC$, and hence we have two questions:

  \begin{enumerate}[label=(\roman*)]
  \item For what $\sC$, does $\FOILS(\sC) = \FOCDS(\sC)$ hold?
  \label{enum: i}
  \item For what $\sC$, does $\FOCDS(\sC) = \FOCLS(\sC)$ hold?
  \label{enum: ii}
  \end{enumerate}

For \ref{enum: i}, we obtain the following theorem by simply extending the theorem in \S~\ref{subsection: Background}.

  \begin{theorem*}
  $\FOCDS(\sC) = \FOCLS(\sC)$ if and only if all connectives in $\sC$ are monotonic.
  \end{theorem*}
The proof is easily obtained by extending that of the theorem in \S~\ref{subsection: Background}. The reader who wants to know a detailed proof is referred to  \cite{matsuda2021effect}.
  
  In this paper, we consider the remaining question, \ref{enum: ii}.
 Before considering the general case, we remark that it is known that in the case of the usual connectives, the presence of disjunction causes the difference between $\FOILS(\sC)$ and $\FOCDS(\sC)$ (cf., e.g., \cite{gabbay1981semantical}). That is, for $\sC \subseteq \{ \lnot, \land, \lor, \to \}$, $\FOILS(\sC) = \FOCDS(\sC)$ if and only if $\lor \notin \sC$. However, this result does not tell us what \emph{property} of disjunction causes the difference. By considering general connectives, our main result clarifies what property of connectives cause the difference between $\FOILS(\sC)$ and $\FOCDS(\sC)$:
   \begin{theorem*}
  $\FOILS(\sC) = \FOCDS(\sC)$ if and only if all $c \in \sC$ are \emph{supermultiplicative}.\footnote{For the definition of supermultiplicativity, see \S\S~\ref{subsection: truth functions}-\ref{subsection: Propositional connectives and formulas}.}
  \end{theorem*}
Furthermore, combining the preceding two theorems, we obtain a necessary and sufficient condition for $\FOILS(\sC)$ to coincide with $\FOCLS(\sC)$:
  \begin{corollary*}
  $\FOILS(\sC) = \FOCLS(\sC)$ if and only if all connectives $c \in \sC$ are both monotonic and supermultiplicative.
  \end{corollary*}
  \begin{remark}
  Although we shall not describe it explicitly, we can easily see from the proofs that these theorems and corollary also hold if we use sequents with a single succedent instead of those with multi-succedents. For example, the proof of the second theorem also shows that $\FOILS_1(\sC) = \FOCDS_1(\sC)$ if and only if all connectives in $\sC$ are supermultiplicative, where $\FOILS_1(\sC) = \{ \Gamma \To \Delta \in \FOILS(\sC) \mid \text{$\Delta$ consists of a single formula} \}$ and $\FOCDS_1(\sC) = \{ \Gamma \To \Delta \in \FOCDS(\sC) \mid \text{$\Delta$ consists of a single formula} \}$.
  \end{remark}  
%%%
%%%
\section{Preliminaries}\label{section: Preliminaries}
\subsection{Truth functions}
\label{subsection: truth functions}
We introduce some notions concerning truth functions briefly. For more detailed presentations, see \S~1.1 in \cite{kawano2021effect}.

For a sequence $\ba \in \{0,1\}^n$ of truth values of length $n$, we denote by $\ba[i]$ the $i$-th component of $\ba$. We denote $\langle 1, \ldots, 1 \rangle$ and $\langle 0, \ldots, 0 \rangle$ by $\bone$ and $\bzero$, respectively. (The lengths of $\bone$ and $\bzero$ as sequences depend on the context.)
An $n$-ary \emph{truth function} is a function from $\{ 0, 1\}^n$ to $\{ 0, 1\}$.

We define a partial order $\sqsubseteq$ on $\{ 0, 1\}^n$ by $\ba \sqsubseteq \bb \iff \text{$\ba[i] \leq \bb[i]$ for all $1 \leq i \leq n$}$.
For $\ba, \bb \in \{0,1\}^n$, $\ba \sqcap \bb$ denotes the infimum of $\{ \ba, \bb \}$, so that $(\ba \sqcap \bb)[i] = \min (\ba[i], \bb[i])$ for all $1 \leq i \leq n$.
An $n$-ary truth function $f$ is said to be \emph{supermultiplicative} if for all $\ba, \bb \in \{ 0,1\}^{\arity(c)}$, $f(\ba) = f(\bb) = 1$ implies $f(\ba \sqcap \bb) = 1$.\footnote{Note that $f(\ba) = f(\bb) = 1 \Rightarrow f(\ba \sqcap \bb) = 1$ if and only if $f(\ba) \sqcap f(\bb) \leq f(\ba \sqcap \bb)$.}
\subsection{Formulas and sequents}
\label{subsection: Propositional connectives and formulas}
A \emph{(propositional) connective} is a symbol with a truth function. For a propositional connective $c$, we denote by $\ttfunc{c}$ the truth function associated with $c$ and by $\arity(c)$ the arity of $\ttfunc{c}$.
A connective is said to be \emph{supermultiplicative} if its truth function is supermultiplicative. Note that it can easily be verified that for a connective $c$ with $\arity(c) \leq 2$, $c$ is supermultiplicative if and only if $c$ is neither disjunction $\lor$ nor exclusive disjunction $\exor$.\footnote{The truth function of exclusive disjunction is defined as follows: $\ttfunc{\exor}(x,y)=1$ if and only if $x \neq y$.}

Fix a set $\sC$ of propositional connectives.
An \emph{atomic formula} is an expression of the form $p(x_1, \ldots, x_n)$, where $p$ is an $n$-ary predicate symbol and $x_1, \ldots, x_n$ are individual variables.\footnote{Although we assume the language has no function symbols and no constant symbols, all arguments in this paper work as well in case the language has function symbols and constant symbols.}
The set $\Fml(\sC)$ of \emph{(predicate) formulas} with propositional connectives in $\sC$ is defined inductively as follows:
  \begin{itemize}[itemsep=1pt]
  \item if $\alpha$ is an atomic formula, then $\alpha \in \Fml(\sC)$;
  \item if $c \in \sC$ and $\alpha_1, \ldots, \alpha_{\arity(c)} \in \Fml(\sC)$, then $c(\alpha_1, \ldots, \alpha_{\arity(c)}) \in \Fml(\sC)$;
  \item if $\alpha \in \Fml(\sC)$ and $x$ is an individual variable, then $\forall x \alpha \in \Fml(\sC)$ and $\exists x \alpha \in \Fml(\sC)$.
  \end{itemize}
We denote by $\FV(\alpha)$ the set of free variables of $\alpha$. In particular, $\FV(c(\alpha_1, \ldots, \alpha_{\arity(c)}))$ is defined to be $\FV(\alpha_1) \cup \cdots \cup \FV(\alpha_{\arity(c)})$.

A \emph{(predicate) sequent} is an expression of the form $\Gamma \To \Delta$, where $\Gamma$ and $\Delta$ are sets of formulas. We denote by $\FOSqt(\sC)$ the set $\{ \Gamma \To \Delta \mid \Gamma, \Delta \subseteq \FOFml(\sC) \}$. If $\Gamma = \{ \alpha_1, \ldots, \alpha_n \}$ and $\Delta = \{ \beta_1, \ldots, \beta_m\}$, we often omit the braces and simply write $\alpha_1, \ldots, \alpha_n \To \beta_1, \ldots, \beta_m$ for $\Gamma \To \Delta$.

\subsection{Kripke Semantics}
\emph{Kripke models} are defined the same as in the case of first-order intuitionistic logic with the usual propositional connectives. That is,
a Kripke model is a tuple $\langle W, \preceq, D, I \rangle$, where $\langle W, \preceq \rangle$ is a pre-ordered set; $D$ assigns to each $w \in W$ a non-empty set $D(w)$, called the \emph{individual domain at $w$}; and $I$ is a function, called the \emph{interpretation function}, that assigns to each pair $\langle w, p \rangle$ of a world and an $n$-ary predicate symbol a function $I(w,p)$ from $D(w)^n$ to $\{0,1 \}$. As usual, it is required that $D$ be monotonic and $I$ satisfy the hereditary condition. 
That is, we assume that for any $w, v \in W$, if $w \preceq v$, then $D(w) \subseteq D(v)$ and $I(w, p) (a_1, \ldots, a_n ) \leq I(v,p)(a_1, \ldots, a_n)$ for any $n$-ary predicate symbol and any $a_1, \ldots, a_n \in D(w)$.
An \emph{assignment} in $D(w)$ is a function which assigns to each individual variable an element in $D(w)$. As usual, a Kripke model $\sK = \langle W, \preceq, D, I \rangle$ is said to be \emph{constant domain} if $D(w) = D(v)$ for all $w, v \in W$.

Let $\sK = \langle W, \preceq, D, I \rangle$ be a Kripke model. The \emph{value} $\| \alpha \|_{\sK, w}^\rho \in \{0,1\}$ of a formula $\alpha \in \FOFml(\sC)$ at a world $w \in W$ with respect to an assignment $\rho$ in $D(w)$ is defined as follows:
  \begin{itemize}[itemsep=1pt]
  \item $\| p(x_1, \ldots, x_n) \|_{\sK, w}^\rho = I(w,p)(\rho(x_1), \ldots, \rho(x_n))$;
  \item $\| c(\alpha_1, \ldots, \alpha_{\arity(c)}) \|_{\sK, w}^\rho = 1$ if and only if $\ttfunc{c}(\| \alpha_1 \|_{\sK, v}^{\rho}, \ldots, \| \alpha_{\arity(c)} \|_{\sK, v}^{\rho}) = 1$ for all $v \succeq w$;
  \item $\| \forall x \alpha \|_{\sK, w}^\rho = 1$ if and only if $\| \alpha \|_{\sK, v}^{\rho[x \mapsto a]} = 1$ for all $v \succeq w$ and all $a \in D(v)$;
  \item $\| \exists x \alpha \|_{\sK, w}^\rho = 1$ if and only if $\| \alpha \|_{\sK, w}^{\rho[x \mapsto a]} = 1$ for some $a \in D(w)$,
  \end{itemize}
where $\rho[x \mapsto a]$ is the assignment in $D(w)$ which maps $x$ to $a$ and is equal to $\rho$ everywhere else. As in the case of the usual connectives, the hereditary condition easily extends to any formula:
  \begin{lemma}
  For any $\alpha \in \FOFml(\sC)$, any Kripke model $\sK = \langle W, \preceq, D, I \rangle$, any $w, v \in W$ and any assignment $\rho$ in $D(w)$, if $w \preceq v$ then $\| \alpha \|_{\sK, w}^\rho \leq \| \alpha \|_{\sK, w}^\rho$.
  \end{lemma}
We shall use this lemma without references.

For a Kripke model $\sK = \langle W, \preceq, D, I \rangle$, a possible world $w \in W$, an assignment $\rho$ in $D(w)$ and a sequent $\Gamma \To \Delta \in \FOSqt(\sC)$, the \emph{value} $\| \Gamma \To \Delta \|_{\sK, w}^{\rho} \in \{0,1\}$ of $\Gamma \To \Delta$ at $w$ with respect to $\rho$ is defined by
  \[
  \| \Gamma \To \Delta \|_{\sK, w}^\rho =
  \begin{cases}
  0 & \text{if $\| \alpha \|_{\sK, w}^\rho = 1$ for all $\alpha \in \Gamma$ and $\| \beta \|_{\sK, w}^\rho = 0$ for all $\beta \in \Delta$} \\
1 & \text{otherwise}.
  \end{cases}
  \]
For a Kripke model $\sK = \langle W, \preceq, D, I \rangle$, a sequent $\Gamma \To \Delta \in \FOSqt(\sC)$ is \emph{valid} in $\sK$ (notation: $\sK \vDash \Gamma \To \Delta$) if $\| \Gamma \To \Delta \|_{\sK,w}^\rho = 1$ for all $w \in W$ and all assignment $\rho$ in $D(w)$. A sequent $\Gamma \To \Delta \in \FOSqt(\sC)$ is \emph{Kripke-valid} (resp.~\emph{$\CD$-valid}) if it is valid in all Kripke models (resp.~constant domain Kripke models). We denote by $\FOILS(\sC)$ the set $\{ \Gamma \To \Delta \in \FOSqt(\sC) \mid \text{$\Gamma \To \Delta$ is Kripke-valid} \}$ and by $\FOCDS(\sC)$ the set $\{ \Gamma \To \Delta \in \FOSqt(\sC) \mid \text{$\Gamma \To \Delta$ is $\CD$-valid} \}$. Then, immediately it follows that $\FOILS(\sC) \subseteq \FOCDS(\sC)$ for any $\sC$.

Here, for later use, we prepare some notations.
The value $\| \alpha \|_{\sK, w}^\rho$ depends only on the values of $\rho$ on $\FV(\alpha)$. Hence, for a partial function $\rho$ from the set of individual variables to $D(w)$, $\| \alpha \|_{\sK, w}^\rho$ can be defined if $\rho(x)$ is defined for all $x \in \FV(\alpha)$. Even for such (partial) assignments, we define $\rho[x \mapsto a]$ to be the function which maps $x$ and is equal to $\rho$ on $\dom(\rho) \setminus \{ x \}$. We use $\varnothing$ to denote the empty assignment $\emptyset \to D(w)$. For example, for a Kripke model $\sK = \langle W, \preceq, D, I \rangle$, $w \in W$ and $a, b \in D(w)$, we have $\| p(x, y) \|_{\sK, w}^{\varnothing [x \mapsto a] [y \mapsto b]} = I(w, p) (a,b)$.

If $\vec{\alpha}$ denotes a sequence of formulas $\alpha_1, \ldots, \alpha_n$, then we denote by $\| \vec{\alpha} \|_{\sK, w}^{\rho}$ the sequence of values of $\alpha_1, \ldots, \alpha_n$, $\langle \| \alpha_1 \|_{\sK, w}^\rho, \ldots, \| \alpha_n \|_{\sK, w}^{\rho} \rangle$. For example, if $\vec{\beta} = \beta_1, \ldots, \beta_n$, then $\ttfunc{c}(\| \vec{\beta} \|_{\sK, w}^\rho) = \ttfunc{c}(\| \beta_1 \|_{\sK, w}^\rho, \ldots, \| \beta_n \|_{\sK, w}^\rho)$.
%%%
\section{Condition for $\FOILS(\sC) = \FOCLS(\sC)$} \label{section: Main section}
In this section, we show the following main theorem:
  \begin{theorem}\label{theorem: main theorem}
  $\FOILS(\sC) = \FOCDS(\sC)$ if and only if all connectives in $\sC$ are supermultiplicative.
  \end{theorem}
We show the ``if'' part in \S~\ref{subsection: The ``if'' part} and the ``only if'' part in \S~\ref
{subsection: The ``only if'' part}.
\subsection{The ``if'' part} 
\label{subsection: The ``if'' part}
Here, we show that if all connectives in $\sC$ are supermultiplicative, then $\FOILS(\sC) = \FOCDS(\sC)$. First, for later use, we prepare one lemma concerning supermultiplicativity.
  \begin{lemma}\label{FOLLemma 3}
  If a connective $c$ is supermltiplicative, then $c$ satisfies the following condition:
for all $n \geq 1$ and all $\ba_1, \ldots, \ba_n \in \{0, 1 \}^{\arity(c)}$, if $\ttfunc{c}(\ba_1) = \cdots = \ttfunc{c}(\ba_n) = 1$ then $\ttfunc{c}(\ba_1 \sqcap \cdots \sqcap \ba_n) = 1$.
  \end{lemma}
  \begin{proof}
  This lemma can be shown by easy induction on $n$. \qed
  \end{proof}

 Assume all connectives in $\sC$ are supermultiplicative. Since $\FOILS(\sC) \subseteq \FOCDS(\sC)$ holds, in order to prove $\FOILS(\sC) = \FOCDS(\sC)$,  it suffices to show the converse inclusion, and hence it suffices to show the following claim: if $\sK \nvDash \Gamma \To \Delta$ for some $\Gamma \To \Delta \in \FOSqt(\sC)$ and  some Kripke model $\sK = \langle W, \preceq, D, I \rangle$, then $\sK'' \nvDash \Gamma \To \Delta$ for some constant domain Kripke model $\sK''$. We show this claim by generalizing the method in \cite{gabbay1981semantical}, which is used to prove the claim for the usual connectives: for $\sC \subseteq \{ \lnot, \land, \to \}$, $\FOILS(\sC) = \FOCDS(\sC)$ holds.

Before describing the proof, we introduce some definitions. In a pre-ordered set $\langle A, \preceq \rangle$, a \emph{path} from $a \in A$ is a maximal linear subset of $\{ b \in A \mid b \succeq a \}$. Let $\langle A, \preceq \rangle$ be a pre-ordered set. For $a \in A$ and $B \subseteq A$, we say $B$ \emph{bars} $a$ (notation: $a \barred B$) if, for any path $\sP$ from $a$, $B \cap \sP \neq \emptyset$ holds.

Now, we describe how to transform $\sK$ into $\sK''$. Since $\sK \nvDash \Gamma \To \Delta$, there exist some $w_\star \in W$ and some assignment $\rho_\star$ in $D(w_\star)$ such that $\| \Gamma \To \Delta \|_{\sK, w_\star}^{\rho_\star} = 0$. 
First, we tranform $\sK$ into a tree Kripke model $\sK' = \langle W', \preceq', D', I' \rangle$.\footnote{In \cite{gabbay1981semantical}, $\sK$ is transformed into a Beth model, and then the Beth model is transformed into a constant domain Kripke model. In contrast, we do not introduce a Beth model because it is not necessary for the proof, and instead transform $\sK$ into Kripke model $\sK'$, which plays essentially the same role as the Beth model.}
  \begin{definition}
  Let $\Last$ denote the function which assigns to each non-empty finite sequence of elements of $W$ its last component, so that $\Last(w_0, \ldots, w_n) = w_n$. Then,  Kripke model $\sK' = \langle W', \preceq', D', I' \rangle$ consists of
  \begin{itemize}[itemsep=1pt]
  \item $W' = \{ \langle w_\star, w_1, \ldots, w_n \rangle \mid n \geq 0, w_1, \ldots, w_n \in W, w_\star \preceq w_1 \preceq \cdots \preceq w_n \}$;
  \item $w' \preceq' v'$ if and only if $w'$ is an initial segment of $v'$, that is, $\langle w_\star', w_1, \ldots, w_n \rangle \preceq' \langle w_\star', v_1, \ldots, v_m \rangle$ if and only if $n \leq m$ and $w_i = v_i$ for all $i = 1, \ldots, n$;
  \item $D'(w') = D(\Last(w'))$;
  \item $I'(w') = I(\Last(w'), p)$.
  \end{itemize}
We denote by $w'_\star$ the minimum element $\langle w_\star \rangle$. 
  \end{definition}
Then, it can be shown that $\sK'$ has the following two property (cf.~\cite{kripke1965semantical}):
  \begin{enumerate}[label = (\Roman*)]
  \item for any $\alpha \in \FOFml(\sC)$, any $w' \in W'$ and any assignment $\rho'$ in $D'(w') = D(\Last(w'))$, $\| \alpha \|_{\sK', w'}^{\rho'} = \| \alpha \|_{\sK, \Last(w')}^{\rho'}$;
  \item for any $\alpha \in \FOFml(\sC)$, any $w' \in W'$ and any assignment $\rho'$ in $D'(w')$, $\| \alpha \|_{\sK', w'}^{\rho'} = 1$ if and only if $w' \barred \{ v' \succeq' w' \mid \| \alpha \|_{\sK', v'}^{\rho'} = 1 \}$.
  \label{enum: bar property of K'}
  \end{enumerate}
In particular, by \textnormal{(I)}, it holds that $\| \Gamma \To \Delta \|_{\sK', w_\star'}^{\rho_\star} = \| \Gamma \To \Delta \|_{\sK, w_\star}^{\rho_\star} = 0$. 

Now, we transform $\sK'$ into a constant domain Kripke model $\sK''$ such that $\sK'' \nvDash \Gamma \To \Delta$.
  \begin{definition}\label{FOLDefinition 3}
  Constant domain Kripke model $\sK'' = \langle W', \preceq', D'', I'' \rangle$ consists of
    \begin{itemize}[itemsep=2pt]
    \item the pre-ordered set $\langle W', \preceq' \rangle$ is the same as $\sK'$;
    \item $D''$ is the set of those partial functions $F$ from $W'$ to $\bigcup_{w' \in W'} D'(w')$ which satisfy the following conditions:
      \begin{itemize}[itemsep=2pt]
      \item $w_\star' \barred \dom(F)$ (, where $\dom(F)$ denots the domain of $F$);
      \item $\dom(F)$ is an upward-closed subset of $W'$;
      \item $F(w') \in D'(w')$ for any $w' \in \dom(F)$;
      \item if $\dom(F) \owns w' \preceq' v'$ then $F(w') = F(v')$.
      \end{itemize}
    \item $I''(w', p)(F_1, \ldots, F_n) = 1$ if and only if for any $v' \succeq' w'$, if $v' \in \bigcap_{i = 1, \ldots, n} \dom(F_i)$, then $I'(v', p)(F_1(v'), \ldots, F_n(v')) = 1$. In case of $n = 0$, we promise that $\bigcap_{i = 1, \ldots, n} \dom(F_i)$ denotes $W'$. That is, for a propositional symbol $p$, $I''(w', p) = 1$ if and only if $I'(v', p) = 1$ for all $v' \succeq' w'$.
    \end{itemize}
  \end{definition}
Then, the following lemma can be shown immediately:
  \begin{lemma}
  Let $F_1, \ldots, F_n \in D''$. Then, the followings hold.
    \begin{enumerate}[itemsep=2pt, label=\textnormal{(\roman*)}]
    \item $\bigcap_{i = 1, \ldots, n} \dom(F_i)$ is an upward-closed subset of $W'$.
    \item For any $w' \in W'$, $w' \barred \bigcap_{i = 1,\ldots, n} \dom(F_i)$.
    \end{enumerate}
  \end{lemma}
We shall use this lemma without references.

We prepare two lemmas in order to prove the main lemma, Lemma \ref{lemma: main lemma in if part}.
  \begin{lemma}\label{FOLLemma 6}
  Let $V'$ be an upward-closed subset of $W'$. Then, there is a family $\{ V_i' \mid i \in I \}$ of subsets of $V'$ such that
    \begin{itemize}
    \item $\bigcup_{i \in I} V_i' = V'$;
    \item for all $i, j \in I$, $V_i' \cap V_j' = \emptyset$ if $i \neq j$;
    \item for all $i \in I$, $V_i'$ is an upward-closed subset of $W'$;
    \item for all $i \in I$, $V_i'$ has a minimum element.
    \end{itemize}
  In particular, if $w' \in W'$ and $V' = \{ v' \in W' \mid v' \succeq' w'\}$, then we can take $\{ V' \}$ as $\{ V_i' \mid i \in I \}$.
  \end{lemma}
We call $\{V_i' \mid i \in I \}$ a \emph{partition} of $V'$.
  \begin{proof}
  Let $\preceq'_1$ denote the parent--child relation on $W'$, that is, for $w', v' \in W'$, $w' \preceq'_1 v'$ if and only if there exists some $w$ such that $v' = w' * \langle w \rangle$. Let $\sim'$ be the smallest equivalence relation on $V'$ that includes the restriction of $\preceq'_1$ to $V' \times V'$. Then, we can take as $\{V_i' \mid i \in I \}$ the set of all equivalence classes of $\sim'$. The first three conditions can easily be verified.

  Before showing the last condition, let us consider the infimum of two worlds. For $w' \in W'$ and $v' \in W'$, let $w' \wedge v'$ denote the infimum of $\{ w', v' \}$ with respect to $\preceq'$. We show that if $w', v' \in V_i'$, then $w' \wedge v' \in V_i'$. In order to show this claim by contradiction, suppose $w', v' \in V'_i$ and $w' \wedge v' \notin V_i'$. Then, $w' \wedge v' \in V_j'$ for some $j \neq i$. Since $V'_j$ is upward-closed and $w' \wedge v' \preceq' w'$, we have $w' \in V_j'$, and hence $w' \notin V'_i$, which contradicts $w' \in V_i'$.

  Now we show that $V'_i$ has a minimum element. For the sake of contradiction, suppose $V'_i$ does not have a minimum element. First, since $V_i'$ is non-empty, it has some element $v_0'$. By the supposition, there is a $w' \in V_i'$ such that $v_0' \not \preceq' w'$. Put $v_1' = v_0' \wedge w'$. Then, we can immediately see $v_1' \neq v_0'$, and hence, $v_0' \succ' v'_1$. We can repeat the same process infinitely, and then obtain a descending sequence $v_0' \succ' v_1' \succ' v_2' \succ' \cdots$. However, this contradicts the fact that $W'$ is a tree.
  \qed
  \end{proof}
  \begin{lemma}\label{FOLlem: added lemma}
  Let $S'$ be an upward-closed subset of $W'$ such that $w'_\star \barred S'$, and let $w' \in W'$. Put $V' = \{ v' \in W' \mid v' \succeq' w' \} \cap S'$, and let $\{V'_i \mid i \in I \}$ be an partition of $V'$. Furthermore, for each $i \in I$, let $v'_i$ be the minimum element of $V'_i$.
Suppose $a_i \in D'(v_i')$ for each $i \in I$. Then, there is a $G \in D''$ such that $V' \subseteq \dom(G)$ and $G(v_i') = a_i$ for all $i \in I$.
  \end{lemma}
Note that, since $V'$ is upward-closed, by Lemma \ref{FOLLemma 6}, it has a partition.
  \begin{proof}
  Put $U' = \bigcap_{i \in I} \{ u' \in W' \mid \text{$u'$ is incomparable with $v'_i$} \}$. Then, $U' \cap V' = \emptyset$. First, we show that $U'$ is upward-closed. For the sake of contradiction, suppose $u' \preceq' v'$, $u' \in U'$ and $v' \notin U'$. Then, there is some $i$ such that $v'$ and $v'_i$ are comparable. Thus, either $v'_i \preceq' v'$ or $v' \preceq' v'_i$. If the latter holds, then we have $u' \preceq' v' \preceq' v'_i$, and hence, $u' \notin U'$, which contradicts $u' \in U'$. Hence, we have $v'_i \preceq' v'$. Since $W'$ is a tree, from $v'_i \preceq' v'$ and $u' \preceq' v'$, we can see $v'_i$ and $u'$ are comparable. This contradicts $u' \in U'$. Thus, we have shown that $U'$ is upward-closed. Hence, $V' \cup U'$ is also upward-closed. By Lemma \ref{FOLLemma 6}, there is a partition $\{ U'_j \mid j \in J \}$ of $U'$. For each $j \in J$, let $u'_j$ be the minimum element of $U'_j$. Note that $\{ V'_i \mid i \in I \} \cup \{ U'_j \mid j \in J\}$ is a partition of $V' \cup U'$.

  Now, we show $w'_\star \barred V' \cup U'$. Let $\sP$ be a path from $w'_\star$. First, we consider the case in which $w' \notin \sP$. In this case, we show $U' \cap \sP \neq \emptyset.$ Put $v' = \max \{t' \in \sP \mid t' \preceq' w' \}$. Then, $v' \prec' w'$ follows from $w' \notin \sP$. Since $W'$ is a tree, $\{u' \in W' \mid u' \preceq' v' \} \subseteq \sP$ holds. Furthermore, since $\sP$ is a path, $\{ u' \in W' \mid u' \preceq' v'\} \subsetneq \sP$ holds.  That is, there exists some $u' \in \sP$ such that $u' \not \preceq' v'$. Since $v', u' \in \sP$, $v' \prec' u'$ follows from $u' \not \preceq' v'$. By $w' \notin \sP$, $v' = \max \{ t' \in \sP \mid t' \preceq' w' \}$, $v' \prec' w'$ and $v' \prec' u'$, it can easily be seen that $w'$ and $u'$ are incomparable. Now, we show $u' \in U'$. For the sake of contradiction, suppose $u' \notin U'$. Then, there exists some $i \in I$ such that $v'_i$ and $u'$ are comparable. If $v'_i \preceq' u'$, then $u' \in V'$, and hence, $w' \preceq' u'$, which contradicts the fact that $w'$ and $u'$ are incomparable. Hence, we have $u' \prec' v'_i$. In addition, $w' \preceq' v'_i$ holds by $v'_i \in V'$. Hence, since $W'$ is a tree, $u'$ and $w'$ are comparable. However, this contradicts the fact that $w'$ and $u'$ are incomparable. Thus, our assumption turned out to be false, and hence, we have shown $u' \in U' \cap \sP$ in the case $w' \notin \sP$. Now, we consider the case in which $w' \in \sP$. In this case, we show $V' \cap \sP \neq \emptyset$.
Since $w'_\star \barred S'$, $\sP \cap S'$ has an element, say, $s'$. Since $w', s' \in \sP$, $w'$ and $s'$ are comparable. If $s' \succeq' w'$, then $s' \in V'$. If $w' \succeq' s'$, then, since $S'$ is upward-closed, $w' \in S'$, and hence, $w' \in V'$. Thus, we have completed the proof of $w'_\star \barred V' \cup U'$.

  Taking arbitrary elements $b_j \in D'(u'_j)$ for each $j \in J$, we define a function $G \colon V' \cup U' \to \bigcup_{w' \in W'} D'(w')$ as follows:
    \[
    G(v') =
      \begin{cases}
      a_i & \text{if $v' \in V'_i$ with $i \in I$} \\
      b_j & \text{if $v' \in U'_j$ with $j \in J$}.
      \end{cases}
    \]
  Then, this $G$ is the desired function.
  \qed
  \end{proof}

The following lemma ensures that the values of formulas in $\sK'$ are preserved in some sense. Before describing the lemma, we prepare $\lambda$-notation. Let $V$ be a set of individual variables and $\mathfrak{E}(x)$ an expression of our meta-language that denotes some value for each $x \in V$. Then $\lam{x \in V}\mathfrak{E}(x)$ denotes the function whose domain is $V$ and whose value at each argument $x$ is $\mathfrak{E}(x)$. For example, the expression $\lam{x \in \FV(\alpha)}{\rho''(x)(v')}$ in the following lemma denotes the partial assignment in $D'(v')$ that assigns $\rho''(x)(v') \in D'(v')$ to each $x \in \FV(\alpha)$.
  \begin{lemma}\label{lemma: main lemma in if part}
  Let $\alpha \in \FOFml(\sC)$. Then, the following conditions are equivalent:
    \begin{enumerate}[label = \textnormal{(\roman*)}]
    \item $\| \alpha \|_{\sK'', w'}^{\rho''} = 1$.
    \label{enum: main lemma i}
    \item For any $v' \succeq' w'$, if $v' \in \bigcap_{x \in \FV(\alpha)} \dom(\rho''(x))$ then $\| \alpha \|_{\sK', v'}^{\lam{x \in \FV(\alpha)}{\rho''(x)(v')}} = 1$.
    \label{enum: main lemma ii}
    \end{enumerate}
  If $\FV(\alpha) = \emptyset$, we promise that $\bigcap_{x \in \FV(\alpha)} \dom(\rho''(x))$ denotes $W'$.
  \end{lemma}
Note that if $w' \in \bigcap_{x \in \FV(\alpha)} \dom(\rho''(x))$, then $(\mathrm{ii})$ is equivalent to $\| \alpha \|_{\sK', w'}^{\lam{x \in \FV(\alpha)}{\rho''(x)(w')}} = 1$.

\begin{proof}
The proof proceeds by induction on $\alpha$.\footnote{The proofs of cases 1, 3 and 4 are essentially the same as in \cite{gabbay1981semantical}.}

\begin{description}[itemsep = 5pt, listparindent  = 10pt, leftmargin = 0pt]
\item [Case 1:] $\alpha \equiv p(x_1, \ldots, x_n)$. This case follows from the definition of $I''$.
\item [Case 2:] $\alpha \equiv c(\vec{\beta})$, where $\vec{\beta} = \beta_1, \ldots, \beta_{\arity(c)}$. Put $n = \arity(c)$.
\begin{description}[itemsep = 3pt, listparindent=10pt, leftmargin = 5pt, itemindent=6pt, topsep=2pt]
%\DrawEnumitemLabel 
\item [$\ref{enum: main lemma i} \Rightarrow \ref{enum: main lemma ii}$:] Suppose \ref{enum: main lemma i} holds. Let $v' \succeq' w'$ and $v' \in \bigcap_{x \in \FV(\alpha)}\dom(\rho''(x))$. In order to show $\| c(\vec{\beta})\|_{\sK', v'}^{\lam{x \in \FV(\alpha)}{\rho''(x)(v')}} = 1$, we take an arbitrary $u' \succeq' v'$, and show that $\ttfunc{c}(\| \vec{\beta} \|_{\sK', u'}^{\lam{x \in \FV(\alpha)}{\rho''(x)(v')}}) = 1$. For each $1 \leq i \leq n$, by the induction hypothesis, we have
$\| \beta_i \|_{\sK'', u'}^{\rho''} = \| \beta_i \|_{\sK', u'}^{\lam{x \in \FV(\beta_i)}{\rho''(x)(u')}} = \| \beta_i \|_{\sK', u'}^{\lam{x \in \FV(\beta)}{\rho''(x)(v')}}$. Thus, we have $\ttfunc{c}(\| \vec{\beta} \|_{\sK'', u'}^{\rho''}) = \ttfunc{c}(\| \vec{\beta} \|_{\sK', u'}^{\lam{x \in \FV(\alpha)}{\rho''(x)(v')}})$, the left hand side of which equals to $1$ by the supposition \ref{enum: main lemma i}.
\item [$\ref{enum: main lemma ii} \Rightarrow \ref{enum: main lemma i}$:] Suppose \ref{enum: main lemma ii} holds. We show that $\| c(\vec{\beta})\|_{\sK'', w'}^{\rho''} = 1$. In order to prove this, we suppose $v' \succeq' w'$ and show that $\ttfunc{c}(\| \vec{\beta} \|_{\sK'', v'}^{\rho''}) = 1$. First, we consider the case $\| \vec{\beta} \|_{\sK'', v'}^{\rho''} = \bone$. Then, since $v' \barred \bigcap_{x \in \FV(\alpha)} \dom(\rho''(x))$, there is some $u' \succeq' v'$ such that $u' \in \bigcap_{x \in \FV(\alpha)} \dom(\rho''(x))$. For any $1 \leq i \leq n$, by the hereditary we have $\| \beta_i \|_{\sK'', u'}^{\rho''} \geq \| \beta_i \|_{\sK'', v'}^{\rho''} = 1$, and hence, by the induction hypothesis, we have $\| \beta_i \|_{\sK', u'}^{\lam{x \in \FV(\beta_i)}{\rho''(x)(u')}} = 1$. Thus, we have $\| \vec{\beta} \|_{\sK'', v'}^{\rho''} = \bone = \| \vec{\beta} \|_{\sK', u'}^{\lam{x \in \FV(\alpha)}{\rho''(x)(u')}}$. On the other hand, since $\| \alpha \|_{\sK', u'}^{\lam{x \in \FV(\alpha)}{\rho''(x)(u')}} = 1$ holds by the assumption $(\mathrm{ii})$, we have $\ttfunc{c}(\| \vec{\beta} \|_{\sK',u'}^{\lam{x \in \FV(\alpha)}{\rho''(x)(u')}}) = 1$. Combining these results, we have $\ttfunc{c}(\| \vec{\beta} \|_{\sK'', v'}^{\rho''}) = 1$.

Secondly, we consider the case $\| \vec{\beta} \|_{\sK'', v'}^{\rho''} \neq \bone$. Put $I = \{1 \leq i \leq n \mid \| \beta_i \|_{\sK'', v'}^{\rho''} = 0 \}$ and $J = \{ 1 \leq i \leq n \mid \| \beta_i \|_{\sK'', v'}^{\rho''} = 1 \}$. Note that $I \neq \emptyset$. Now, it suffices to show that for each $i \in I$, there exists a $t'_i \in W'$ such that
$t'_i \in \bigcap_{x \in \FV(\alpha)} \dom(\rho''(x))$;
$\| \beta_i \|_{\sK', t'_i}^{\lam{x \in \FV(\alpha)}{\rho''(x)(t'_i)}} = 0$; $\| \beta_j \|_{\sK', t'_i}^{\lam{x \in \FV(\alpha)}{\rho''(x)}(t'_i)} = 1$ for all $j \in J$; and $\ttfunc{c}(\| \vec{\beta} \|_{\sK' ,t'_i}^{\lam{x \in \FV(\alpha)}{\rho''(x)(t'_i)}}) = 1$. This is because, for such $t'_i$'s, if we take as $\{ \ba_1, \ldots, \ba_n \}$ in Lemma \ref{FOLLemma 3} the set $\{ \| \vec{\beta} \|_{\sK', t'_i}^{\lam{x \in \FV(\alpha)}{\rho''(x)}} \mid i \in I \}$, then it holds that
\[
\ba_1 \sqcap \cdots \sqcap \ba_n = \bigsqcap_{i \in I} \| \vec{\beta} \|_{\sK', t'_i}^{\lam{x \in \FV(\alpha)}{\rho''(x)}} = \| \vec{\beta} \|_{\sK'', v'}^{\rho''},
\] and hence, $\ttfunc{c}(\| \vec{\beta} \|_{\sK'', v'}^{\rho''}) = 1$ follows from Lemma \ref{FOLLemma 3}. So, we fix an arbitrary $i \in I$ and show that such $t'_i \in W'$ exists. By the induction hypothesis for $\beta_i$, there is some $u' \succeq' v'$ such that $u' \in \bigcap_{x \in \FV(\beta_i)}{\dom(\rho''(x))}$ and $\| \beta_i \|_{\sK', u'}^{\lam{x \in \FV(\beta_i)}{\rho''(x)(u')}} = 0$. By the property \ref{enum: bar property of K'} of $\sK'$, $u' \not \barred \{ r' \succeq' u' \mid \| \beta_i \|_{\sK', r'}^{\lam{x \in \FV(\beta_i)}{\rho''(x)(u')}} = 1 \}$ holds. Hence, there is some path $\sP$ from $u'$ such that $\sP \cap \{ r' \succeq' u' \mid \| \beta_i \|_{\sK', r'}^{\lam{x \in \FV(\beta_i)}{\rho''(x)(u')}} =1 \} = \emptyset$. On the other hand, since $u' \barred \bigcap_{x \in \FV(\alpha)} \dom(\rho''(x))$, $\sP$ and $\bigcap_{x \in \FV(\alpha)} \rho''(x)$ intersect at some point, say, $t'_i \in W'$. Then, we have $t'_i \succeq' u' \succeq' v'$, $t'_i \in \bigcap_{x \in \FV(\alpha)}\dom(\rho''(x))$ and 
  \[
  \| \beta_i \|_{\sK', t_i'}^{\lam{x \in \FV(\alpha)}{\rho''(x)(t_i')}} = \| \beta_i \|_{\sK', t_i'}^{\lam{x \in \FV(\alpha)}{\rho''(x)(u')}} = 0.
  \]
 By the hereditary, we have $\| \beta_j \|_{\sK'', t'_i}^{\rho''} \geq \| \beta_j \|_{\sK'', v'}^{\rho''} = 1$ for all $j \in J$. Hence, by the induction hypothesis we have $\| \beta_j \|_{\sK', t'_i}^{\lam{x \in \FV(\alpha)}{\rho''(x)(t'_i)}} = 1$ for all $j \in J$. Finally, since $\| c(\vec{\beta}) \|_{\sK', t'_i}^{\lam{x \in \FV(\alpha)}{\rho''(x)(t'_i)}} = 1$ holds by the assumption \ref{enum: main lemma ii}, we have $\ttfunc{c}(\| \vec{\beta} \|_{\sK', t'_i}^{\lam{x \in \FV(\alpha)}{\rho''(x)(t'_i)}}) = 1$. Thus, we have proved that $t'_i$ satisfies the desired conditions.
\end{description} 
\item [Case 3:] $\alpha \equiv \forall y \beta$. We only consider the case $y \in \FV(\beta)$, since the other case is trivial.
\begin{description}[itemsep = 3pt, listparindent=10pt, leftmargin = 5pt, itemindent=6pt, topsep=2pt]
\item [$\ref{enum: main lemma i} \Rightarrow \ref{enum: main lemma ii}$:] Suppose \ref{enum: main lemma i} holds. We suppose $v' \succeq' w'$ and $v' \in \bigcap_{x \in \FV(\alpha)} \dom(\rho''(x))$, and show $\| \forall y \beta \|_{\sK', v'}^{\lam{x \in \FV(\alpha)}{\rho''(x)(v')}} = 1$, that is, $\| \beta \|_{\sK', u'}^{\left( \lam{x \in \FV(\alpha)}{\rho''(x)(v')} \right)[y \mapsto a]} = 1$ for all $u' \succeq v'$ and all $a \in D'(u')$. Let $u' \succeq' v'$ and $a \in D'(u')$.
Then, by Lemma \ref{FOLlem: added lemma}, with $w'$, $S'$ and $\{V'_i \mid i \in I \}$ there taken to be $u'$, $W'$ and $\{ \{ t' \in W' \mid t' \succeq' u' \} \}$, there is some $G \in D''$ such that $G(u') = a$.
By the hypothesis $\| \forall y \beta \|_{\sK'', w'}^{\rho''} = 1$, we have $\| \beta \|_{\sK'', u'}^{\rho''[y \mapsto G]} = 1$. By the induction hypothesis, we have $\| \beta \|_{\sK', u'}^{\lam{x \in \FV(\beta)}{\rho''[y \mapsto G](x)(u')}} = 1$. Since
\[
\left( \lam{x \in \FV(\alpha)}{\rho''(x)(v')} \right) [y \mapsto a] = \lam{x \in \FV(\beta)}{\rho''[y \mapsto G](x)(u')},
\]
$\| \beta \|_{\sK', u'}^{\left( \lam{x \in \FV(\alpha)}{\rho''(x)(v')} \right) [y \mapsto a]} = 1$ follows.
\item [$\ref{enum: main lemma ii} \Rightarrow \ref{enum: main lemma i}$:] Suppose \ref{enum: main lemma ii} holds. We show $\| \alpha \|_{\sK'', w'}^{\rho''} = 1$, that is, $\| \beta \|_{\sK'', w'}^{\rho''[y \mapsto G]} = 1$ for all $G \in D''$. Let $G \in D''$. Then, by the induction hypothesis, $\| \beta \|_{\sK'', w'}^{\rho''[y \mapsto G]} = 1$ if and only if, for all $v' \succeq' w'$, $v' \in \bigcap_{x \in \FV(\alpha)} \dom(\rho''(x)) \cap \dom(G)$ implies $\| \beta \|_{\sK', v'}^{\lam{x \in \FV(\beta)}{\rho''[y \mapsto G](x)(v')}} = 1$.
Hence, in order to show $\| \beta \|_{\sK'', w'}^{\rho''[y \mapsto G]} = 1$, we suppose $v' \succeq' w'$ and $v' \in \bigcap_{x \in \FV(\alpha)} \dom(\rho''(x)) \cap \dom(G)$ and show $\| \beta \|_{\sK', v'}^{\lam{x \in \FV(\beta)}{\rho''[y \mapsto G](x)(v')}} = 1$. By \ref{enum: main lemma ii}, $\| \beta \|_{\sK', v'}^{\left( \lam{x \in \FV(\alpha)}{\rho''(x)(v')} \right) [y \mapsto G(v')]} = 1$. Thus, we have
\[
\| \beta \|_{\sK', v'}^{\lam{x \in \FV(\beta)}{\rho''[y \mapsto G](x)(v')}} = \| \beta \|_{\sK', v'}^{\left( \lam{x \in \FV(\alpha)}{\rho''
(x)(v')} \right) [y \mapsto G(v')]} = 1 .
\]
\end{description}
\item [Case 4:] $\alpha \equiv \exists y \beta$. We only consider the case $y \in \FV(\beta)$, since the other case is trivial.
\begin{description}[itemsep = 3pt, listparindent=10pt, leftmargin = 5pt, itemindent=6pt, topsep=2pt]
\item [$\ref{enum: main lemma i} \Rightarrow \ref{enum: main lemma ii}$:] Suppose \ref{enum: main lemma i} holds. Then, there is some $G \in D''$ such that $\| \beta \|_{\sK'', w'}^{\rho''[y \mapsto G]} = 1$. Suppose $v' \succeq' w'$ and $v' \in \bigcap_{x \in \FV(\alpha)} \dom(\rho''(x))$, in order to show that $\| \exists y \beta \|_{\sK', v'}^{\lam{x \in \FV(\alpha)}{\rho''(x)(v')}} = 1$. By the property \ref{enum: bar property of K'} of $\sK'$, it suffices to show $v' \barred \{ u' \succeq' v' \mid \| \exists y \beta \|_{\sK', u'}^{\lam{x \in \FV(\alpha)}{\rho''(x)(v')}} = 1\}$. Let $\sP$ be any path from $v'$. Then, since $v' \barred \bigcap_{x \in \FV(\beta)} \dom(\rho''[y \mapsto G](x))$, there is some $u' \in \sP \cap \bigcap_{x \in \FV(\beta)} \dom(\rho''[y \mapsto G](x))$.
By the induction hypothesis, $\| \beta \|_{\sK', u'}^{\lam{x \in \FV(\beta)}{\rho''[y \mapsto G](x)(u')}} = 1$ follows from $\| \beta \|_{\sK'', w'}^{\rho''[y \mapsto G]} = 1$.
Hence, we have
\[
\| \beta \|_{\sK', u'}^{\bigl( \lam{x \in \FV(\alpha)}{\rho''(x)(u')} \bigr) [y \mapsto G(u')]} = \| \beta \|_{\sK', u'}^{\lam{x \in \FV(\beta)}{\rho''[y \mapsto G](x) (u')}} = 1.
\]
Hence, we have $\| \exists y \beta \|_{\sK', u'}^{\lam{x \in \FV(\alpha)}{\rho''(x)(v')}} = \| \exists y \beta \|_{\sK', u'}^{\lam{x \in \FV(\alpha)}{\rho''(x)(u')}} = 1$. Thus, we have proved $v' \barred \{ u' \succeq' v' \mid \| \exists y \beta \|_{\sK', u'}^{\lam{x \in \FV(\alpha)}{\rho''(x)(v')}} = 1\}$.
\item [$\ref{enum: main lemma ii} \Rightarrow \ref{enum: main lemma i}$:] Suppose \ref{enum: main lemma ii} holds. Put $V':= \{v' \in W' \mid v' \succeq' w' \} \cap \bigcap_{x \in \FV(\alpha)} \dom(\rho''(x))$.
By Lemma \ref{FOLLemma 6} , $V'$ has a partition $\{ V'_i \mid i \in I \}$. For each $i \in I$, let $v'_i$ be the minimum element of $V'_i$.Then, by \ref{enum: main lemma ii} we have $\| \exists y \beta \|_{\sK', v_i'}^{\lam{x \in \FV(\alpha)}{\rho''(x)(v'_i)}} = 1$ for each $i \in I$. Hence, for each $i \in I$, there is some $a_i \in D'(v'_i)$ such that $\| \beta \|_{\sK', v'_i}^{\bigl( \lam{x \in \FV(\alpha)}{\rho''(x)(v'_i)} \bigr) [y \mapsto a _i]} = 1$. By Lemma \ref{FOLlem: added lemma}, there is some $G \in D''$ such that $V' \subseteq \dom(G)$ and $G(v_i') = a_i$ for all $i \in I$. In order to prove the goal, $\| \exists y \beta \|_{\sK', w'}^{\rho''} = 1$, we show that $\| \beta \|_{\sK'', w'}^{\rho''[y \mapsto G]} = 1$. By the induction hypothesis, it suffices to show $\| \beta \|_{\sK', v'}^{\lam{x \in \FV(\beta)}{\rho''[y \mapsto G](x)(v')}} = 1$ for any $v' \succeq' w'$ with $v' \in \bigcap_{x \in \FV(\beta)} \dom(\rho''[y \mapsto G](x))$. Furthermore, since
\[
\{ u' \in W' \mid u' \succeq' w' \} \cap \bigcap_{x \in \FV(\beta)} \dom(\rho''[y \mapsto G](x)) = V' \cap \dom(G) = V',
\]
it suffices to show $\| \beta \|_{\sK', v'}^{\lam{x \in \FV(\beta)}{\rho''[y \mapsto G](x)(v')}} = 1$ for any $v' \in V'$.
Moreover, by the definition of $v'_i$'s, it suffices to show $\| \beta \|_{\sK', v_i'}^{\lam{x \in \FV(\beta)}{\rho''[y \mapsto G](x)(v'_i)}} = 1$ for all $i \in I$. By the definition of $G$, immediately we have $\lam{x \in \FV(\beta)}{\rho''[y \mapsto G](x)(v_i')} = \bigl( \lam{x \in \FV(\alpha)}{\rho''(x) (v'_i)} \bigr) [y \mapsto a_i]$ for all $i \in I$. Hence, for all $i \in I$, we have
\[
\| \beta \|_{\sK', v'_i}^{\lam{x \in \FV(\beta)}{\rho''[y \mapsto G](x)(v'_i)}} = \| \beta \|_{\sK', v'_i}^{\bigl( \lam{x \in \FV(\alpha)}{\rho''(x)(v'_i)} \bigr)[y \mapsto a_i]} = 1.
\]
\end{description}
\end{description}
\qed
\end{proof}

From Lemma \ref{lemma: main lemma in if part}, it follows that $\| \Gamma \To \Delta \|_{\sK'', w'_\star}^{\rho''} = \| \Gamma \To \Delta \|_{\sK', w'_\star}^{\rho_\star} = 0$, where $\rho''$ is the assignment in $D''$ such that for each free variables $x$ in $\Gamma$ and $\Delta$, $\rho''(x)$ is the function on $W'$ whose value is constantly $\rho_\star(x) \in D'(w'_\star)$. Thus, we have finished the proof of the ``if'' part of Theorem \ref{theorem: main theorem}.

%%%%%%%%%%%%%%%%%%%%%%%%%%%%%%%%%%%%%%%

\subsection{The ``only if'' part}
\label{subsection: The ``only if'' part}
Here, we show the ``only if'' part of Theorem \ref{theorem: main theorem} by showing its contraposition:
  \begin{proposition}\label{proposition: contraposition of the only if part}
  Suppose there is a non-supermultiplicative $c \in \sC$.
  Then, it holds that $\FOILS(\sC) \neq \FOCDS(\sC)$.
  \end{proposition}
First, let us consider the case in which $\arity(c) \leq 2$. Since, as mentioned in \S~\ref{subsection: Propositional connectives and formulas}, non-supermultiplicative connectives whose arity is less than or equal to $2$ are only $\lor$ and $\exor$ (exclusive disjunction), we only have to consider the cases $c = \lor$ and $c = \exor$. Regarding disjunction, it is known that the sequents of the form $\forall x (p(x) \lor q(x)) \To \forall x p(x) \lor \exists x q(x)$ are CD-valid but not Kripke-valid (cf., e.g., \cite{nagashima1973intermediate}). As to $\exor$, we can verify that the corresponding sequents of the form $\forall x (p(x) \exor q(x)) \To \forall x p(x) \exor \exists x q(x)$ are also $\CD$-valid but not Kripke-valid (cf. case (A) in the proof below).\footnote{In contrast, we can verify that the sequents corresponding to $\D$-axioms, $\forall x (p(x) \exor r) \To \forall p(x) \exor r$, where $r$ is a $0$-ary predicate symbol, are Kripke-valid.} For $c$ of general airty, we construct a sequent in $\FOCDS(\sC) \setminus \FOILS(\sC)$ which plays the same role as $\forall x (p(x) \lor q(x)) \To \forall x p(x) \lor \exists x q(x)$ and $\forall x (p(x) \exor q(x)) \To \forall x p(x) \exor \exists x q(x)$. This construction requires an elaborate case analysis. 
\begin{proof}
Suppose there is a non-supermultiplicative connective $c \in \sC$. Then, there are $\ba, \bb \in \{ 0, 1 \}^{\arity(c)}$ such that $\ttfunc{c}(\ba) = \ttfunc{c}(\bb) = 1$ and $\ttfunc{c}(\ba \sqcap \bb) = 0$.
Let $p$ and $q$ be distinct unary predicate symbols. Fix two propositional symbols $T$ and $R$, which shall play particular roles in the proof.

First, we define a Kripke model $\sK^\star = \langle W^\star, \preceq^\star, D^\star, I^\star \rangle$, which is used to show the sequents constructed below are not Kripke-valid.
\begin{itemize}
\item $W^\star = \{ w_1, w_2 \}$;
\item $w_i \preceq^\star w_j$ if and only if $i \leq j$;
\item $D^\star(w_1) = \{a_1\}$, $D^\star(w_2) = \{ a_1, a_2 \}$;
\item
\begin{itemize}
\item $I(w_1, p) (a_1)= 1$, $I(w_1, q)(a_1) = 0$, $I(w_1, T) = 1$, $I(w_1, R) = 0$;
\item $I(w_2, p) (a_1) = 1$, $I(w_2, q) (a_1) = 0$, $I(w_1, T) = 1$, $I(w_2, R) = 0$, \\
$I(w_2, p) (a_2) = 0$, $I(w_2, q)(a_2) = 1$;
\end{itemize}
\end{itemize}
We define $\ba^*, \bb^* \in \{0,1\}^{\arity(c)}$ as follows:
\begin{align*}
\ba^*[i] & =
\begin{cases}
\ba[i] & \text{if $\ba[i] = 1$ or $\bb[i] = 1$} \\
1 & \text{if $\ba[i] = 0$ and $\bb[i] =0$},
\end{cases} \\
\bb^*[i] & =
\begin{cases}
\bb[i] & \text{if $\ba[i] =1$ or $\bb[i] = 1$} \\
1 & \text{if $\ba[i] = 0$ and $\bb[i] = 0$}.
\end{cases}
\end{align*}
Then, we have $\ba \sqcap \bb^* = \ba^* \sqcap \bb = \ba \sqcap \bb$ and $\ba \sqcup \bb^* = \ba^* \sqcup \bb = \ba^* \sqcup \bb^* = \bone$. We divide into five cases: (A) $\ttfunc{c}(\bone) = 0$; (B) $\ttfunc{c}(\ba^*) = 1$; (C) $\ttfunc{c}(\bb^*) = 1$; (D) $\ttfunc{c}(\ba^*) = \ttfunc{c}(\bb^*) = \ttfunc{c}(\ba^* \sqcap \bb^*) = 0$ and $\ttfunc{c}(\bone) = 1$; (E) $\ttfunc{c}(\ba^*) = \ttfunc{c}(\bb^*) = 0$ and $\ttfunc{c}(\ba^* \sqcap \bb^*) = \ttfunc{c}(\bone) = 1$.

\begin{description}[itemsep = 5pt, listparindent  = 10pt, leftmargin=0pt]
\item [Case \textnormal{(A)}:] $\ttfunc{c}(\bone) =0$. Define $F \in \FOFml(\sC)$ by $F \equiv c(T, \ldots, T)$. Note that, for any Kripke model $\sK = \langle W, \preceq, D, I \rangle$ and any $w \in W$, if $\| T \|_{\sK, w}^{\emptyfunc} = 1$ then $\| F \|_{\sK, v}^{\emptyfunc} = 0$ for all $v \succeq w$.

We define formulas $\varphi_1, \ldots, \varphi_{\arity(c)}, \varphi, \psi_1, \ldots, \psi_{\arity(c)}, \psi \in \FOFml(\sC)$ as follows:
\begin{align*}
\varphi_i & \equiv
\begin{cases}
F & \text{if $\ba[i] = 0$ and $\bb[i] = 0$} \\
p(x) & \text{if $\ba[i] = 0$ and $\bb[i] = 1$} \\
q(x) & \text{if $\ba[i] = 1$ and $\bb[i]$ = 0} \\
T & \text{if $\ba[i] = 1$ and $\bb[i] =1$},
\end{cases} \\
\varphi & \equiv \forall x c (\varphi_1, \ldots, \varphi_{\arity(c)}), \\
\psi_i & \equiv
\begin{cases}
F & \text{if $\ba[i] = 0$ and $\bb[i] = 0$} \\
\forall x p(x) & \text{if $\ba[i] = 0$ and $\bb[i] = 1$} \\
\exists x q(x) & \text{if $\ba[i] = 1$ and $\bb[i]$ = 0} \\
T & \text{if $\ba[i] = 1$ and $\bb[i] =1$},
\end{cases} \\
\psi & \equiv c(\psi_1, \ldots, \psi_{\arity(c)}).
\end{align*}
Put $\vec{\varphi} = \varphi_1, \ldots, \varphi_{\arity(c)}$ and $\vec{\psi} = \psi_1, \ldots, \psi_{\arity(c)}$. We show $T, \varphi \To \psi \in \FOCDS(\sC) \setminus \FOILS(\sC)$.

First, we show $T, \varphi \To \psi \in \FOCDS(\sC)$. Let $\sK = \langle W, \preceq, D, I \rangle$ be a constant domain Kripke model and $w \in W$. We suppose $\| T \|_{\sK, w}^\emptyfunc = \| \varphi \|_{\sK, w}^{\emptyfunc} = 1$, in order to show $\| \psi \|_{\sK, w}^{\emptyfunc} = 1$. To show $\| \psi \|_{\sK, w}^\emptyfunc = 1$, it suffices to show $\ttfunc{c}(\| \vec{\psi} \|_{\sK, v}^\emptyfunc) = 1$ for any $v \succeq w$. Let $v \succeq w$. Then, we can see $\| \forall x p(x) \|_{\sK, v}^\emptyfunc = 1$ or $\| \exists x q(x) \|_{\sK, v}^\emptyfunc = 1$ holds. For otherwise there exists some $a \in D$ such that $\| p(x) \|_{\sK, v}^{\emptyfunc[x \mapsto a]} = \| q(x) \|_{\sK, v}^{\emptyfunc[x \mapsto a]} = 0$, and hence, $\| \vec{\varphi} \|_{\sK, v}^{\emptyfunc[x \mapsto a]} = \ba \sqcap \bb$, and thus, $\ttfunc{c}(\| \vec{\varphi} \|_{\sK, v}^{\emptyfunc[x \mapsto a]}) = 0$, which contradicts $\| \varphi \|_{\sK, w}^\emptyfunc = 1$.
First, we consider the case $\| \forall x p(x) \|_{\sK, v}^\emptyfunc = \| \exists x q(x) \|_{\sK, v}^\emptyfunc = 1$. Then we have $\| \vec{\psi} \|_{\sK, v}^\emptyfunc = \ba \sqcup \bb$ and there exists some $a \in D$ such that $\| p(x) \|_{\sK, v}^{\emptyfunc [x \mapsto a]} = \| q(x) \|_{\sK, v}^{\emptyfunc[x \mapsto a]} = 1$, so that $\| \vec{\varphi} \|_{\sK, v}^{\emptyfunc [x \mapsto a]} = \ba \sqcup \bb$. Hence, we have $\ttfunc{c}(\| \vec{\psi} \|_{\sK, v}^\emptyfunc) = \ttfunc{c}(\ba \sqcup \bb) = \ttfunc{c}(\| \vec{\varphi} \|_{\sK, v}^{\emptyfunc [x \mapsto a]})$ the right hand side of which equals to $1$ by $\| \varphi \|_{\sK, w}^\emptyfunc = 1$. Thus, $\ttfunc{c}(\| \vec{\psi} \|_{\sK, v}^\emptyfunc) = 1$. Next, we consider the case that one of $\| \forall x p(x) \|_{\sK, v}^\emptyfunc$ and $\| \exists x q(x) \|_{\sK, v}^\emptyfunc$ is $1$ and the other is $0$. Then, either $\| \vec{\psi} \|_{\sK, v}^\emptyfunc = \ba$ or $\| \vec{\psi} \|_{\sK, v}^\emptyfunc = \bb$, and hence, we have $\ttfunc{c}(\| \vec{\psi} \|_{\sK, v}^\emptyfunc) = 1$.

Secondly, we show $T, \varphi \To \psi \notin \FOILS(\sC)$. In order to do so, we verify $\| T, \varphi \To \psi \|_{\sK^\star, w_1}^{\emptyfunc} = 0$. First, we can easily see the followings:
$\| \vec{\varphi} \|_{\sK^\star, w_1}^{\emptyfunc[x \mapsto a_1]} = \| \vec{\varphi} \|_{\sK^\star, w_2}^{\emptyfunc[x \mapsto a_1]} = \bb$; $\| \vec{\varphi} \|_{\sK^\star, w_2}^{\emptyfunc[x \mapsto a_2]} = \ba$; $\| \forall x p(x) \|_{\sK^\star, w_1}^\emptyfunc = \| \exists x q(x) \|_{\sK^\star, w_1}^{\emptyfunc} = 0$; and $\| \vec{\psi} \|_{\sK^\star, w_1}^{\emptyfunc} = \ba \sqcap \bb$. From these it follows that $\| \varphi \|_{\sK^\star, w_1}^\emptyfunc = 1$ and $\| \psi \|_{\sK^\star, w_1}^\emptyfunc = 0$.
\item [Case \textnormal{(B)}:] $\ttfunc{c}(\ba^*) = 1$. Note that, in this case, there is no $i$ such that $\ba^*[i] = \bb[i] = 0$. We define formulas $\varphi_1, \ldots, \varphi_{\arity(c)}, \varphi, \psi_1, \ldots, \psi_{\arity(c)}, \psi \in \FOFml(\sC)$ as follows:
\begin{align*}
\varphi_i & \equiv
\begin{cases}
p(x) & \text{if $\ba^*[i] = 0$ and $\bb[i] = 1$} \\
q(x) & \text{if $\ba^*[i] = 1$ and $\bb[i] = 0$} \\
T & \text{if $\ba^*[i] = 1$ and $\bb^*[i] = 1$}
\end{cases} \\
\varphi & \equiv \forall x c (\varphi_1, \ldots, \varphi_{\arity(c)}) \\
\psi_i & \equiv
\begin{cases}
\forall x p(x) & \text{if $\ba^*[i] = 0$ and $\bb[i] = 1$} \\
\exists x q(x) & \text{if $\ba^*[i] = 1$ and $\bb[i] = 0$} \\
T & \text{if $\ba^*[i] = 1$ and $\bb^*[i] = 1$}
\end{cases} \\
\psi & \equiv c(\psi_1, \ldots, \psi_{\arity(c)})
\end{align*}
Put $\vec{\varphi} = \varphi_1, \ldots, \varphi_{\arity(c)}$ and $\vec{\psi} = \psi_1, \ldots, \psi_{\arity(c)}$. We show $T, \varphi \To \psi \in \FOCDS(\sC) \setminus \FOILS(\sC)$.

First, $T, \varphi \To \psi \in \FOCDS(\sC)$ can be proved similarly to case (A); in fact, if we replace every $\ba$ in the proof in case (A) by $\ba^*$ (and thus, $\ba \sqcap \bb$ by $\ba^* \sqcap \bb$ and $\ba \sqcup \bb$ by $\ba^* \sqcup \bb$), then we obtain a proof for case (B).

$T, \varphi \To \psi \notin \FOILS(\sC)$ can also be proved similarly to case (A); in fact, if we replace every $\ba$ in the proof in case (A) by $\ba^*$ (and thus, $\ba \sqcap \bb$ by $\ba^* \sqcap \bb$), then we obtain a proof for case (B).
\item [Case \textnormal{(C)}:] $\ttfunc{c}(\bb^*) = 1$. This case can be shown similarly to case (B).

\item [Case \textnormal{(D)}:] $\ttfunc{c}(\ba^*) = \ttfunc{c}(\bb^*) = \ttfunc{c}(\ba^* \sqcap \bb^*) = 0$ and $\ttfunc{c}(\bone) = 1$. Note that, in this case, since $\ttfunc{c}(\ba^*) \neq \ttfunc{c}(\ba)$, we have $\ba^* \neq \ba$, and hence, there is at least one $i$ such that $\ba[i] = \bb[i] = 0$. We define formulas $\varphi_1, \ldots, \varphi_{\arity(c)}, \varphi, \psi_1, \ldots, \psi_{\arity(c)}, \psi \in \FOFml(\sC)$ as follows:
\begin{align*}
\varphi_i & \equiv
\begin{cases}
R & \text{if $\ba[i] = 0$ and $\bb[i] = 0$} \\
p(x) & \text{if $\ba[i] = 0$ and $\bb[i] = 1$} \\
q(x) & \text{if $\ba[i] = 1$ and $\bb[i]$ = 0} \\
T & \text{if $\ba[i] = 1$ and $\bb[i] =1$}
\end{cases} \\
\varphi & \equiv \forall x c(\varphi_1, \ldots, \varphi_{\arity(c)}) \\
\psi_i & \equiv
\begin{cases}
R & \text{if $\ba[i] = 0$ and $\bb[i] = 0$} \\
\forall x p(x) & \text{if $\ba[i] = 0$ and $\bb[i] = 1$} \\
\exists x q(x) & \text{if $\ba[i] = 1$ and $\bb[i]$ = 0} \\
T & \text{if $\ba[i] = 1$ and $\bb[i] =1$}
\end{cases} \\
\psi & \equiv c(\psi_1, \ldots, \psi_{\arity(c)})
\end{align*}
Put $\vec{\varphi} = \varphi_1, \ldots, \varphi_{\arity(c)}$ and $\vec{\psi} = \psi_1, \ldots, \psi_{\arity(c)}$. We show $T, \varphi \To \psi \in \FOCDS(\sC) \setminus \FOILS(\sC)$.

First, we show $T, \varphi \To \psi \in \FOCDS(\sC)$. Let $\sK = \langle W, \preceq, D, I \rangle$ be a constant domain Kripke model and $w \in W$. We suppose $\| T \|_{\sK, w}^\emptyfunc = \| \varphi \|_{\sK, w}^\emptyfunc = 1$, in order to show $\| \psi \|_{\sK, w}^\emptyfunc = 1$. To show $\| \psi \|_{\sK, w}^\emptyfunc = 1$, it suffices to show $\ttfunc{c}(\| \vec{\psi} \|_{\sK, v}^\emptyfunc) = 1$ for any $v \succeq w$. Let $v \succeq w$. Then, we divide into two subcases according to the value of $\| R \|_{\sK, v}^\emptyfunc$.
  \begin{description}[itemsep = 3pt, listparindent=10pt, itemindent=6pt]
  \item [Subcase $(\mathrm{i})$:] $\| R \|_{\sK, v}^\emptyfunc = 0$. In this case, $\| \vec{\psi} \|_{\sK, v}^\emptyfunc = 1$ can be shown similarly to case (A) because $R$ plays the same role as $F$ in case (A).
  \item [Subcase $(\mathrm{ii})$:] $\| R \|_{\sK, v}^\emptyfunc = 1$. In this case, $\| \forall x p(x) \|_{\sK, v}^\emptyfunc = \| \exists x q(x) \|_{\sK, v}^\emptyfunc = 1$ holds. For, otherwise, there is some $a \in D$ such that $\| \vec{\varphi} \|_{\sK, v}^{\emptyfunc [x \mapsto a]}$ equals to either $\ba^*$, $\bb^*$ or $\ba^* \sqcap \bb^*$, and hence, $\ttfunc{c}(\| \vec{\varphi} \|_{\sK, v}^{\emptyfunc [x \mapsto a]}) = 0$, which contradicts $\| \varphi \|_{\sK, w}^\emptyfunc = 1$. Thus, we have $\| \vec{\psi} \|_{\sK, v}^\emptyfunc = \bone$. Hence, $\ttfunc{c}(\| \vec{\psi} \|_{\sK, v}^\emptyfunc) = 1$.
  \end{description}
Secondly, we show $T, \varphi \To \psi \notin \FOILS(\sC)$.
Since $\| R \|_{\sK^\star, w_1}^\emptyfunc = \| R \|_{\sK^\star, w_2}^\emptyfunc = 0$, $R$ plays the same role as $F$ in case (A). Thus, $T, \varphi \To \psi \in \FOILS(\sC)$ can be shown similarly to case (A).
\item [Case \textnormal{(E)}:] $\ttfunc{c}(\ba^*) = \ttfunc{c}(\bb^*) = 0$ and $\ttfunc{c}(\ba^* \sqcap \bb^*) = \ttfunc{c}(\bone) = 1$. Note that, in this case, there is no $i$ such that $\ba[i] = \bb[i] = 0$. First, in order to construct the desired sequent, we define biconditional $\leftrightarrow$ ($\ttfunc{\leftrightarrow}(x,y) = 1$ if and only if $x = y$) using $c$. That is, for any $\alpha \in \FOFml(\sC)$ and any $\beta \in \FOFml(\sC)$, we define a formula $\alpha \leftrightarrow_c \beta \in \FOFml(\sC)$. First, we define $\theta_1^{\alpha, \beta}, \ldots, \theta_{\arity(c)}^{\alpha, \beta}$ for $\alpha, \beta \in \FOFml(\sC)$ by
\[
\theta_i^{\alpha, \beta} \equiv
\begin{cases}
\alpha & \text{if $\ba^*[i] = 0$, $\bb^*[i] = 1$} \\
\beta & \text{if $\ba^*[i] = 1$, $\bb^*[i] = 0$} \\
T & \text{if $\ba^*[i] = 1$, $\bb^*[i] = 1$}.
\end{cases}
\]
Put $\overrightarrow{\theta^{\alpha, \beta}} \equiv \theta_1^{\alpha, \beta}, \ldots, \theta_{\arity(c)}^{\alpha, \beta}$.
We define $\alpha \leftrightarrow_c \beta$ by $\alpha \leftrightarrow_c \beta \equiv c(\theta_1^{\alpha, \beta}, \ldots, \theta_{\arity(c)}^{\alpha, \beta})$. Then, $\leftrightarrow_c$ has the same meaning as biconditional $\leftrightarrow$ whenever the value of $T$ is interpreted as $1$, that is, for any Kripke model $\sK = \langle W, \preceq, D, I \rangle$, any $w \in W$ and any assignment $\rho$ in $D(w)$, if $\| T \|_{\sK, w}^\rho = 1$, then $\| \alpha \leftrightarrow_c \beta \|_{\sK, w}^\rho = 1$ holds if and only if $\| \alpha \|_{\sK, v}^\rho = \| \beta \|_{\sK, v}^\rho$ holds for all $v \succeq w$. In order to show the ``if'' part, let $\sK = \langle W, \preceq, D, I \rangle$ be a Kripke model and $w \in W$, and suppose $\| T \|_{\sK, w}^\rho = 1$ and $\| \alpha \|_{\sK, v}^\rho = \| \beta \|_{\sK, v}^\rho$ for all $v \succeq w$. Then, for any $v \succeq w$, $\| \overrightarrow{\theta^{\alpha, \beta}} \|_{\sK, v}^\rho$ is either $\ba^* \sqcap \bb^*$ or $\bone$, and hence, $\ttfunc{c}(\| \overrightarrow{\theta^{\alpha, \beta}}\|_{\sK, v}^\rho) = 1$ for any $v \succeq w$. Thus, we have $\| \alpha \leftrightarrow_c \beta \|_{\sK, w}^\rho = 1$. In order to show the (contraposition of) ``only if'' part, let $\sK = \langle W, \preceq, D, I \rangle$ be a Kripke model and $w \in W$, and suppose $\| T \|_{\sK, w}^\rho = 1$ and $\| \alpha \|_{\sK, v}^\rho \neq \| \beta \|_{\sK, v}^\rho$ for some $v \succeq w$. Then, $\| \overrightarrow{\theta^{\alpha, \beta}} \|_{\sK, v}^\rho$ is either $\ba^*$ or $\bb^*$, and hence, we have $\ttfunc{c}(\| \overrightarrow{\theta^{\alpha, \beta}} \|_{\sK, v}^\rho) = 0$. Thus, we have $\| \alpha \leftrightarrow_c \beta \|_{\sK, w}^\rho = 0$.

Now, we define $\varphi_1, \ldots, \varphi_{\arity(c)}, \varphi, \psi_1, \ldots, \psi_{\arity(c)}. \psi \in \FOFml(\sC)$ as in case (D). Put $\vec{\varphi} = \varphi_1, \ldots, \varphi_{\arity(c)}$ and $\vec{\psi} = \psi_1, \ldots, \psi_{\arity(c)}$. We show $T, R \leftrightarrow_c \forall x p(x), R \leftrightarrow_c \exists x q(x), \varphi \To \psi \in \FOCDS(\sC) \setminus \FOILS(\sC)$.

First, we show $T, R \leftrightarrow_c \forall x p(x), R \leftrightarrow_c \exists x q(x), \varphi \To \psi \in \FOCDS(\sC)$. Let $\sK = \langle W, \preceq, D, I \rangle$ be a constant domain Kripke and $w \in W$. We suppose $\| T \|_{\sK, w}^\emptyfunc = \| P \leftrightarrow_c \forall x p(x) \|_{\sK, w}^\emptyfunc = \| P \leftrightarrow_c \forall x q(x) \|_{\sK, w}^\emptyfunc = \| \varphi \|_{\sK, w}^\emptyfunc = 1$, in order to show $\| \psi \|_{\sK, w}^\emptyfunc = 1$. To show $\| \psi \|_{\sK, w}^\emptyfunc = 1$, it suffices to show $\ttfunc{c}(\| \vec{\psi} \|_{\sK, v}^\emptyfunc) = 1$ for any $v \succeq w$. Let $v \succeq w$. Then, we divide into two cases according to the value of $\| R \|_{\sK, v}^\emptyfunc$.
  \begin{description}[itemsep = 3pt, listparindent=10pt, itemindent=6pt]
  \item [Subcase $(\mathrm{i})$:] $\| R \|_{\sK, v}^\emptyfunc = 0$. In this case, $\| \vec{\psi} \|_{\sK, v}^\emptyfunc = 1$ can be shown similarly to subcase $(\mathrm{i})$ in case (D).
  \item [Subcase $(\mathrm{ii})$:] $\| R \|_{\sK, v}^\emptyfunc = 1$. Since $\leftrightarrow_c$ has the same meaning as $\leftrightarrow$ whenever the value of $T$ is interpreted as $1$, $\| \forall x p(x) \|_{\sK, v}^\emptyfunc = \| \forall x q(x) \|_{\sK, v}^\emptyfunc = 1$ follows from $\| T \|_{\sK, w}^\emptyfunc = 1$, $\| R \leftrightarrow_c \forall x p(x) \|_{\sK, w}^\emptyfunc = \| R \leftrightarrow_c \forall x q(x) \|_{\sK, w}^\emptyfunc = 1$. Hence, we have $\| \forall x p(x) \|_{\sK, v}^\emptyfunc = \| \exists x q(x) \|_{\sK, v}^\emptyfunc = 1$, and thus, $\| \vec{\psi} \|_{\sK, v}^\emptyfunc = \bone$. Hence, we have $\ttfunc{c}(\| \vec{\psi} \|_{\sK, v}^\emptyfunc) = 1$.
  \end{description}

Secondly, $T, R \leftrightarrow_c \forall x p(x), R \leftrightarrow_c \exists x q(x), \varphi \To \psi \notin \FOILS(\sC)$ can be shown similarly to case (D).
\end{description}
\qed
\end{proof}

%%%
\section{Conclusion}\label{Section: Conclusion}
 We have extended generalized Kripke semantics to first-order logic. By considering general propositional connectives, we have clarified what property of connectives causes the difference between intuitionistic first-order logic and the logic of constant domains in terms of validity of sequents. Furthermore, as mentioned in \S~\ref{subsection: Predicate logic with truth functional connectives}, we have also found out what property of connectives causes the difference between the logic of constant domains and classical logic and the difference between intuitionistic logic and classical logic. The results are summarized as follows:
   \begin{theorem*} Let $\sC$ be a set of propositional connectives. Then, the followings hold.
     \begin{itemize}
     \item $\FOILS(\sC) = \FOCDS(\sC)$ if and only if all connectives in $\sC$ are supermultiplicative.
     \item $\FOCDS(\sC) = \FOCLS(\sC)$ if and only if all connectives in $\sC$ are monotonic.
     \item $\FOILS(\sC) = \FOCLS(\sC)$ if and only if all connectives in $\sC$ are both monotonic and supermultiplicative.
     \end{itemize}
   \end{theorem*}

\bibliographystyle{splncs04}
\bibliography{main}
\end{document}